\documentclass[preprint,1p]{elsarticle}
\makeatletter
\def\ps@pprintTitle{%
\let\@oddhead\@empty
\let\@evenhead\@empty
\let\@oddfoot\@empty
\let\@evenfoot\@oddfoot
}
\usepackage{graphics}
\usepackage{epsfig}
\usepackage{mathptmx}
\usepackage{amsmath}
\usepackage{amssymb}
\usepackage{hyperref}
\usepackage{fullpage}
\usepackage{amsthm}
\usepackage{lineno}
\journal{}
\theoremstyle{plain}
\newtheorem{theorem}{Theorem}[section]

\newtheorem{lemma}[theorem]{Lemma}

\theoremstyle{definition}
\newtheorem{definition}{Definition}[section]
\theoremstyle{example}

\theoremstyle{remark}
\newtheorem{remark}{Remark}[section]
\numberwithin{equation}{section}
\begin{document}
\begin{frontmatter}
\title{Bivariate-Schurer-Stancu operators based on $(p,q)$-integers}
\author[]{Abdul Wafi}
\ead{awafi@jmi.ac.in}
\author[]{Nadeem Rao\corref{cor1}}
\ead{nadeemrao1990@gmail.com}
\address[]{Department of Mathematics, Jamia Millia Islamia, New Delhi-110 025, India}
\cortext[cor1]{Corresponding author}
\begin{abstract}
The aim of this article is to introduce a bivariate extension of Shurer-Stancu operators based on $(p,q)$-integers. We prove uniform approximation by means of Bohman-Korovkin type theorem, rate of convergence using total modulus of smoothness and degree of approximation by using second order modulus of smoothness, Peetre's K-functional, Lipschitz type class.
\end{abstract}
\end{frontmatter}

\section{Introduction}\label{intr}

In 1962, Schurer \cite{schurer} introduced the following generalization of the classical Bernstein operators for all non-negative integer $l$ and $n\in \mathbb{N}$
\begin{eqnarray*}
B_n^l(f;x)=\sum\limits_{k=0}^{n+l}{n+l \choose k}x^k(1-x)^{n+l-k}f\Big(\frac{k}{n}\Big),
\end{eqnarray*}
where  $f\in C[0,l+1]$ and $x\in [0,1]$.

In recent past, the applications of $q$-calculus attracted the attention of mathematicians and has an interesting impact in the research in approximation theory. It has been noticed that linear positive operators based on $q$-integers are quite effective as far as the rate of convergence is concerned. In 1987, Lupas \cite{Lupas} first defined $q$-analogue of Bernstein operators. In 1997, Philips \cite{Philips} studied other form of Bernstein-polynomials based on $q$-integers. Several extensions of $q$-linear positive operators have been studied by different researchers for instance (see \cite{vijay}). Recently, Mursaleen et al \cite{MB} added an idea based on $(p,q)$-calculus in approximation theory and gave a $(p,q)$ extension to the classical Bernstein operators. The motive of $(p,q)$-integers was to generalize various forms of $q$-oscillator algebras in physics \cite{chakar}. Several generalization of Bernstein operators were studied using $(p,q)$-analogue and their approximation properties have been investigated. For instance, $(p,q)$-Bernstein-Stancu operators \cite{BST};  $(p,q)$-Bernstein-Schurer operators \cite{BS}; $(p,q)$-Schurer-Stancu operators \cite{SS}; Chlodowski variant of $(p,q)$-Schurer-Stancu operators \cite{CSS}; $(p,q)$-Szasz operators \cite{S}; $(p,q)$-Szasz-Baskakov operators \cite{SB} were introduced and their approximation properties are studied. Motivated by the above generalizations, we present a bivariate extension of $(p,q)$-Schurer-Stancu operators in this paper.

Let $0<q<p\leq 1$. Then, $(p,q)$-integers for non negative integers $n,k$ are given by
\begin{eqnarray*}
[k]_{p,q}=\frac{p^k-q^k}{p-q}\hspace{1 cm} and\hspace{0.5 cm} [k]_{p,q}=1\hspace{1 cm} for\hspace{0.5 cm} k=0.
\end{eqnarray*}
 $(p,q)$-binomial coefficient
\begin{eqnarray*}
{n \choose k}_{p,q}=\frac{[n]_{p,q}!}{[k]_{p,q}![n-k]_{p,q}!},
\end{eqnarray*}
and $(p,q)$-binomial expansion
\begin{eqnarray*}
(ax+by)_{p,q}^{n}&=&\sum\limits_{k=0}^{n} {n \choose k}_{p,q}p^{\frac{(n-k)(n-k-1)}{2}}q^{\frac{k(k-1)}{2}}a^{n-k}b^kx^{n-k}y^k,\\
(x+y)_{p,q}^{n}&=&(x+y)(px+py)(p^2x+q^2y)...(p^{n-1}x-q^{n-1}y).
\end{eqnarray*}
\section{\textbf{Construction of $(p,q)$-Bivariate-Schurer-Stancu operators }}
Let $I=[0,l+1]$ and $(x_1,x_2)\in I\times I=[0,l+1]\times[0,l+1]$. Then, for any function $f\in C(I\times I)$ and $(n_1,n_2)\in \mathbb{N}\times \mathbb{N} $, the operators $S_{n_{12},l}^{\alpha_{12},\beta_{12}}:C(I\times I)\rightarrow C([0,1]\times [0,1] )$ is defined as follows
\begin{eqnarray}
\label{eq1}S_{n_{12},l}^{\alpha_{12},\beta_{12}}(f;p_{12},q_{12};x_1,x_2)=\sum_{\nu_1=0}^{n_1+l}\sum_{\nu_2=0}^{n_2+l}s_{n_1,l,\nu_1}^{p_1,q_1}(x_1)s_{n_2,l,\nu_2}^{p_2,q_2}(x_2)f\Bigg(\frac{p^{n_1-\nu_1}[\nu_1]_{p_1q_1}+\alpha_1}{[n_1]_{p_1q_1}+\beta_1},\frac{p^{n_2-\nu_2}[\nu_2]_{p_2q_2}+\alpha_2}{[n_2]_{p_2q_2}+\beta_2}\Bigg),
\end{eqnarray}
where $S_{n_{12},l}^{\alpha_{12},\beta_{12}}(f;p_{12},q_{12},;x_1,x_2)=S_{n_{1},n_{2},l_{1},l_{2}}^{\alpha_{1},\alpha_{2},\beta_{1},\beta_{2}}(f;p_1,q_1,p_2,q_2;x_1,x_2)$ and
\begin{eqnarray*}
 s_{n_i,l,\nu_i}^{{p_i,q_i}}(x_i)=\frac{1}{p^{\frac{(n_i+l)(n_i+l-1)}{2}}}{n_i+l \choose \nu_i}_{p_i,q_i}p^{\frac{\nu_i(\nu_i-1)}{2}}x^{\nu_i}\prod\limits_{j=0}^{n_i+l-\nu_i-1}(p_i^j-q_i^jx_i),
\end{eqnarray*}
with the conditions
\newline
(i) for any positive real number $p_i$ and $q_i$ $(i=1,2)$ such that $0<q_i<p_i\leq 1$,
\newline
(ii) for any non-negative real value of $\alpha_i$ and $\beta_i$ $(i=1,2)$ such that $0\leq\alpha_i\leq\beta_i$.
\newline

\begin{remark}
 One can find that\\
(i) if $p_i=1 (i=1,2)$, then the operators defined by \ref{eq1} reduce to $q$-Bivariate-Schurer-Stancu operators,
\begin{eqnarray*} S_{n_{12},l}^{\alpha_{12},\beta_{12}}(f;q_1,q_2;x_1,x_2)=\sum_{\nu_1=0}^{n_1+l}\sum_{\nu_2=0}^{n_2+l}s_{n_1,l,\nu_1}^{q_1}(x_1)s_{n_2,l,\nu_2}^{q_2}(x_2)f\Bigg(\frac{[\nu_1]_{q_1}+\alpha_1}{[n_1]_{q_1}+\beta_1},\frac{[\nu_2]_{q_2}+\alpha_2}{[n_2]_{q_2}+\beta_2}\Bigg),
\end{eqnarray*}
where
\begin{eqnarray*}
s_{n_i,l,\nu_i}^{q_i}(x_i)={n_i+l \choose \nu_i}_{q_i}x^{\nu_i}\prod\limits_{j=0}^{n_i+l-\nu_i-1}(1-q_i^jx_i),
\end{eqnarray*}

(ii) if $\alpha_i=\beta_i=0,(i=1,2),$ then the operators defined by \ref{eq1} reduce to $(p,q)$-Bivariate-Bernstein-Schurer operators
\begin{eqnarray*}
S_{n_{12},l}^{\alpha_{12},\beta_{12}}(f;p_{12},q_{12};x_1,x_2)=\sum_{\nu_1=0}^{n_1+l}\sum_{\nu_2=0}^{n_2+l}s_{n_1,l,\nu_1}^{p_1,q_1}(x_1)s_{n_2,l,\nu_2}^{p_2,q_2}(x_2)f\Bigg(\frac{p^{n_1-\nu_1}[\nu_1]_{p_1q_1}}{[n_1]_{p_1q_1}},\frac{p^{n_2-\nu_2}[\nu_2]_{p_2q_2}}{[n_2]_{p_2q_2}}\Bigg),
\end{eqnarray*}
and
\newline
(iii) if $l_i=0$ and $\alpha_i=\beta_i=0,(i=1,2),$ then the operators defined by \ref{eq1} reduce to $(p,q)$-Bivariate-Bernstein operators
\begin{eqnarray*}
S_{n_{12}}^{\alpha_{12},\beta_{12}}(f;p_{12},q_{12};x_1,x_2)=\sum_{\nu_1=0}^{n_1}\sum_{\nu_2=0}^{n_2}s_{n_1,\nu_1}^{p_1,q_1}(x_1)s_{n_2,\nu_2}^{p_2,q_2}(x_2)f\Bigg(\frac{p^{n_1-\nu_1}[\nu_1]_{p_1q_1}}{[n_1]_{p_1q_1}},\frac{p^{n_2-\nu_2}[\nu_2]_{p_2q_2}}{[n_2]_{p_2q_2}}\Bigg),
\end{eqnarray*}
where
\begin{eqnarray*}
 s_{n_i,\nu_i}^{{p_i,q_i}}(x_i)=\frac{1}{p^{\frac{(n_i)(n_i-1)}{2}}}{n_i \choose \nu_i}_{p_i,q_i}p^{\frac{\nu_i(\nu_i-1)}{2}}x^{\nu_i}\prod\limits_{j=0}^{n_i-\nu_i-1}(p_i^j-q_i^jx_i).
\end{eqnarray*}
\end{remark}

\begin{lemma}\label{lem1.1}
 Let $e_{i,j}=x_1^ix_2^j, 0\leq i,j \leq 2$ are the two dimensional test functions. Then, we have
 \begin{eqnarray*}
 S_{n_{12},l}^{\alpha_{12},\beta_{12}}(e_{0,0};p_{12},q_{12};x_1,x_2)&=&1,\\
 S_{n_{12},l}^{\alpha_{12},\beta_{12}}(e_{10};p_{12},q_{12};x_1,x_2)&=&\frac{[n_1+l]x_1+\alpha_1}{[n_1]+\beta_1},\\
S_{n_{12},l}^{\alpha_{12},\beta_{12}}(e_{01};p_{12},q_{12};x_1,x_2)&=&\frac{[n_2+l]x_2+\alpha_2}{[n_2]+\beta_2},\\
S_{n_{12},l}^{\alpha_{12},\beta_{12}}(e_{11};p_{12},q_{12};x_1,x_2)&=&\frac{[n_1+l]x_1+\alpha_1}{[n_1]+\beta_1}\frac{[n_2+l]x_2+\alpha_2}{[n_2]+\beta_2},\\
S_{n_{12},l}^{\alpha_{12},\beta_{12}}(e_{20};p_{12},q_{12};x_1,x_2)&=&\frac{[n_1+l](p_1^{n_1+l-1}+2\alpha_1)x_1}{([n_1]+\beta_1)^2}+\frac{q_1[n_1+l][n_1+l-1]x_1^2}{([n_1]+\beta_1)^2}+\frac{\alpha_1^2}{([n_1]+\beta_1)^2},\\
S_{n_{12},l}^{\alpha_{12},\beta_{12}}(e_{02};p_{12},q_{12};x_1,x_2)&=&\frac{[n_2+l](p_2^{n_2+l-1}+2\alpha_2)x_2}{([n_2]+\beta_2)^2}+\frac{q_2[n_2+l][n_2+l-1]x_2^2}{([n_2]+\beta_2)^2}+\frac{\alpha_2^2}{([n_2]+\beta_2)^2}.
 \end{eqnarray*}
 \end{lemma}

\begin{proof}
From equation \ref{eq1}, we find that
 \begin{eqnarray*}
S_{n_{12},l}^{\alpha_{12}\beta_{12}}(e_{0,0};p_{12},q_{12};x_1,x_2)&=&S_{n_{1},l_{1}}^{\alpha_{1},\beta_{1}}(e_0;p_1,q_1;x_1)S_{n_{2},l_{2}}^{\alpha_{2},\beta_{2}}(e_0;p_{2},q_{2};x_2),\\
S_{n_{12},l}^{\alpha_{12}\beta_{12}}(e_{1,0};p_{12},q_{12};x_1,x_2)&=&S_{n_{1},l_{1}}^{\alpha_{1},\beta_{1}}(x_1;p_1,q_1;x_1)S_{n_{2},l_{2}}^{\alpha_{2},\beta_{2}}(e_0;p_{2},q_{2};x_2),\\
S_{n_{12},l}^{\alpha_{12}\beta_{12}}(e_{0,1};p_{12},q_{12};x_1,x_2)&=&S_{n_{1},l_{1}}^{\alpha_{1},\beta_{1}}(e_0;p_1,q_1;x_1)S_{n_{2},l_{2}}^{\alpha_{2},\beta_{2}}(x_2;p_{2},q_{2};x_2),\\
S_{n_{12},l}^{\alpha_{12}\beta_{12}}(e_{1,1};p_{12},q_{12};x_1,x_2)&=&S_{n_{1},l_{1}}^{\alpha_{1},\beta_{1}}(x_1;p_1,q_1;x_1)S_{n_{2},l_{2}}^{\alpha_{2},\beta_{2}}(x_2;p_{2},q_{2};x_2),\\
S_{n_{12},l}^{\alpha_{12}\beta_{12}}(e_{2,0};p_{12},q_{12};x_1,x_2)&=&S_{n_{1},l_{1}}^{\alpha_{1},\beta_{1}}(x_1^2;p_1,q_1;x_1)S_{n_{2},l_{2}}^{\alpha_{2},\beta_{2}}(e_0;p_{2},q_{2};x_2),\\  S_{n_{12},l}^{\alpha_{12}\beta_{12}}(e_{0,2};p_{12},q_{12};x_1,x_2)&=&S_{n_{1},l_{1}}^{\alpha_{1},\beta_{1}}(e_0;p_1,q_1;x_1)S_{n_{2},l_{2}}^{\alpha_{2},\beta_{2}}(x_2^2;p_{2},q_{2};x_2),
 \end{eqnarray*}
using these equalities, we can simply proof the Lemma \ref{lem1.1}.
\end{proof}

\begin{lemma}\label{lem2.2}
Let $S_{n_{12},l}^{\alpha_{12},\beta_{12}}(f;p_{12},q_{12},;x_1,x_2)$ be defined by \ref{eq1}. Then, we have
\begin{eqnarray*}
S_{n_{12},l}^{\alpha_{12}\beta_{12}}(((t_1-x_1)^2);p_{12},q_{12};x_1,x_2)&=&\frac{[n_1+l][n_1+l-1]q_1-2[n_1+l]([n_1]+\beta_1)+([n_1]+\beta_1)^2}{([n_1]+\beta_1)^2}x_1^2\\
&&+\frac{[n_1+l](p_1^{n_1+l-1}+2\alpha_1)-2\alpha_1([n_1]+\beta_1)}{([n_1]+\beta_1)^2}x_1+\frac{\alpha_1^2}{(n+\beta_1)^2}.\\
S_{n_{12},l}^{\alpha_{12}\beta_{12}}(((t_2-x_2)^2);p_{12},q_{12};x_1,x_2)&=&\frac{[n_2+l][n_2+l-1]q_2-2[n_2+l]([n_2]+\beta_2)+([n_2]+\beta_2)^2}{([n_2]+\beta_2)^2}x_2^2\\
&&+\frac{[n_2+l](p_2^{n_2+l-1}+2\alpha_2)-2\alpha_2([n_2]+\beta_2)}{([n_2]+\beta_2)^2}x_2+\frac{\alpha_2^2}{(n+\beta_2)^2}.
 \end{eqnarray*}
  \end{lemma}
\begin{proof}
In view of Lemma \ref{lem1.1} and linearity property, it is easy to prove Lemma \ref{lem2.2}.
\end{proof}

\label{sec2}\section{Main Results}

\begin{definition}
Let $ X,Y \subset \mathbb{R}$ be any two given intervals and the set $B(X\times Y)=\{f:X\times Y\rightarrow \mathbb{R}| f$ is bounded on $X\times Y \}$. For $f\in B(X\times Y)$, let the function $\omega_{total}(f;\cdot,*):[0,\infty)\times [0,\infty)\rightarrow \mathbb{R}$, defined for any $(\delta_1,\delta_2)\in [0,\infty)\times [0,\infty)$ by
\begin{eqnarray*}
 \omega_{total}(f;\delta_1,\delta_2)=sup\{|f(x,y)-f(x',y')|:(x,y),(x',y')\in [0,\infty)\times [0,\infty), |x-x'|\leq \delta_1, |y-y'|\leq \delta_2\},
\end{eqnarray*}
is called the first order modulus of smoothness of the function $f$ or the total modulus of continuity of the function $f$.
\end{definition}

In order to get the rate of convergence and degree of approximation for the operators $S_{n_{12},l}^{\alpha_{12}\beta_{12}}$, we consider $p_i=p_{n_i}$ and $q_i=q_{n_i}$ for $i=1,2$ such that $0<q_{n_i}<p_{n_i}\leq 1$ satisfying
\begin{eqnarray}\label{eq2}
\lim\limits_{n_i\rightarrow \infty} p_{n_i}^{n_i}\rightarrow a_i, \lim\limits_{n_i\rightarrow \infty} q_{n_i}^{n_i}\rightarrow b_i \textrm{ and }\lim\limits_{n_i\rightarrow \infty} p_{n_i}\rightarrow 1, \lim\limits_{n_i\rightarrow \infty} q_{n_i}\rightarrow 1 (i=1,2).
\end{eqnarray}

Here, we recall the following result due to Volkov \cite{volkov}:

\begin{theorem}\label{thm2.1}
  Let $I$ and $J$ be compact intervals of the real line. Let $L_{n_1,n_2}:C(I\times J)\rightarrow C(I\times J), (n_1,n_2)\in \mathbb{N}\times \mathbb{N}$ be linear positive operators. If
  \begin{eqnarray*}
\lim\limits_{n_1,n_2\rightarrow \infty}L_{n_1,n_2}(e_{ij})&=&e_{ij},(i,j)\in \{(0,0),(1,0),(0,1)\},\\
\lim\limits_{n_1,n_2\rightarrow \infty}L_{n_1,n_2}(e_{20}+e_{02})&=&e_{20}+e_{02},
  \end{eqnarray*}
  uniformly on $I\times J$, then the sequence $(L_{n_1,n_2}f)$ converges to $f$ uniformly on $I\times J$ for any $f\in C(I\times J)$.
\end{theorem}

\begin{theorem}\label{thm2.2}
Let $e_{ij}(x_1,x_2)=x_1^ix_2^j(0\leq i+j\leq 2, i,j\in \mathbb{N})$ be the test functions defined on $I\times I$ and $(p_{n_i}),(q_{n_i}),i=1,2$ be the sequences defined by \ref{eq2}. If
\begin{eqnarray*}
\lim\limits_{n_1,n_2\rightarrow \infty}(S_{n_{12},l}^{\alpha_{12}\beta_{12}}e_{ij})(x_1,x_2)=e_{ij}(x_1,x_2),
\end{eqnarray*}
uniformly on $I\times I$, then
\begin{eqnarray*}
\lim\limits_{n_1,n_2\rightarrow \infty}(S_{n_{12},l}^{\alpha_{12}\beta_{12}}f)(x_1,x_2)=f(x_1,x_2).
\end{eqnarray*}
uniformly for any $f\in C(I\times I)$.
\end{theorem}

\begin{proof}
  Using the Theorem \ref{thm2.1} and Lemma \ref{lem1.1}, Theorem \ref{thm2.2} can easily be proved.
\end{proof}

\begin{theorem}\label{thm3.2}\cite{stancu}
 Let $L: C([0,\infty)\times [0,\infty))\rightarrow B([0,\infty)\times [0,\infty))$ be a linear positive operator. For any $f\in C(X\times Y)$, any $(x,y)\in X\times Y$ and any $\delta_1,\delta_2>0$, the following inequality
\begin{eqnarray*}
|(Lf)(x,y)-f(x,y)|&\leq& |Le_{0,0}(x,y)-1||f(x,y)|+\Big[Le_{0,0}(x,y)+\delta_1^{-1}\sqrt{Le_{0,0}(x,y)(L(\cdot-x^2))^2(x,y)}\\
&+&\delta_2^{-1}\sqrt{Le_{0,0}(x,y)(L(*-y^2))^2(x,y)}+\delta_1^{-1}\delta_2^{-1}\sqrt{(Le_{0,0})^2(x,y)(L(\cdot-x^2))^2(x,y)(L(*-y^2))^2(x,y)}\Big]\\
&\times & \omega_{total}(f;\delta_1,\delta_2),
\end{eqnarray*}
holds.
\end{theorem}

 \begin{theorem}\label{thm3.4}
   Let $f\in C(I\times I)$ and $(x_1,x_2)\in I\times I$. Then, for $(n_1,n_2)\in \mathbb{N}$ and for any $\delta_1,\delta_2>0$, we have
\begin{eqnarray*}
|(S_{n_{12},l}^{\alpha_{12}\beta_{12}}f)(x_1,x_2)-f(x_1,x_2)|\leq 4 \omega_{total}(f;\delta_1,\delta_2),
\end{eqnarray*}
where ${\delta_{n_{12},l}^{\alpha_{12}\beta_{12}}}(x_i)=\sqrt{S_{n_{12},l}^{\alpha_{12}\beta_{12}}(((t_i-x_i)^2);p_{12},q_{12};x_1,x_2)}$.
 \end{theorem}

\textbf{Proof} Using Theorem \ref{thm3.2} and Lemma \ref{lem2.2} , we can arrive at the proof of the Theorem \ref{thm3.4}.
\section{Local approximations}
Let $C_B^2(I)=\{f\in C_B(I): f^{(i,j)}\in C_B(I), 1\leq i,j\leq 2\},$ where $C_B(I)$ is the space of all bounded and uniformly continuous functions on $I$ and $f^{i,j}$ is $(i,j)^{th}$-order of partial derivative with respect to $x,y$ of $f$, endowed with the norm
\begin{eqnarray*}
\parallel f\parallel_{C_B^2(I)}=\| f\parallel_{C_B(I)}+\sum\limits_{i=1}^{2}\Big\| \frac{\partial^i f}{\partial x_1^i}\Big\| +\sum\limits_{i=1}^{2}\Big\| \frac{\partial^i f}{\partial x_2^i}\Big\|.
\end{eqnarray*}
The Peetre's K-functional of the function $f\in C_B(I)$ is given by
\begin{eqnarray}\label{k}
K(f;\delta)=\inf\limits_{g\in C_B(I)^2}\{\|f-g\|_{C_B(I)}+\delta\|g\|_{C_B(I)^2},\delta>0\}.
\end{eqnarray}
The following inequality
\begin{eqnarray*}
K(f;\delta)\leq M_1\{\omega_2(f;\sqrt{\delta})+min(1,\delta)\|f\|_{C_B(I)}\},
\end{eqnarray*}
holds for all $\delta>0$ where $M_1$ is a constant independent of $\delta$ and $f$ and $\omega_2(f;\sqrt{\delta})$ is the second order modulus of continuity which is defined in a similar manner as the second order modulus of continuity for one variable case
\begin{eqnarray*}
\omega(f;\sqrt{\delta})=\sup\limits_{0<h\leq\sqrt{\delta}x}\sup\limits_{x+2h\in [0,\infty)}|f(x+2h)-2f(x+h)+f(x)|.
\end{eqnarray*}

\begin{theorem}
Let $(q_{n_i})$ and $(p_{n_i})$ for $i=1,2$ are the real sequences defined in \ref{eq2}. Then, for $f\in C_B^2(I\times I)$, we have the following
\begin{eqnarray*}
|S_{n_{12},l}^{\alpha_{12}\beta_{12}}(g;x_1,x_2)-f(x_1,x_2)|&\leq& 4K(f;M_{n_1,n_2}(x_1,x_2))\\
&+&\omega\Bigg(f;\sqrt{\Big(\frac{[n_1+l]x_1+\alpha_1}{[n_1]+\beta_1}\Big)^2+\Big(\frac{[n_2+l]x_2+\alpha_2}{[n_2]+\beta_2}\Big)^2}\Bigg)\\
&\leq& M\Bigg\{\omega_2\Big(f:\sqrt(M_{n_1,n_2}(x_1,x_2))\Big)+min\{1,M_{n_1,n_2}(x_1,x_2)\}\|f\|_{C_B^2(I)}\Bigg\}\\
&+&\omega\Bigg(f;\sqrt{\Big(\frac{[n_1+l]x_1+\alpha_1}{[n_1]+\beta_1}\Big)^2+\Big(\frac{[n_2+l]x_2+\alpha_2}{[n_2]+\beta_2}\Big)^2}\Bigg),
\end{eqnarray*}
where $M_{n_1,n_2}(x_1,x_2)=\big(\delta_{n_{12},l}^{\alpha_{12}\beta_{12}}(x_1)\big)^2+\big(\delta_{n_{12},l}^{\alpha_{12}\beta_{12}}(x_2)\big)^2$.
\end{theorem}

\begin{proof}
   Consider the auxiliary operators
\begin{eqnarray}\label{eq3}
\widehat{S_{n_{12},l}^{\alpha_{12}\beta_{12}}}(f;x_1,x_2)=S_{n_{12},l}^{\alpha_{12}\beta_{12}}(f;x_1,x_2)-f\Bigg(\frac{[n_1+l]x_1+\alpha_1}{[n_1]+\beta_1},\frac{[n_2+l]x_2+\alpha_2}{[n_2]+\beta_2}\Bigg)+f(x_1,x_2).
\end{eqnarray}
We get that,
\begin{eqnarray}
\widehat{S_{n_{12},l}^{\alpha_{12}\beta_{12}}}(f;x_1,x_2)\leq 3\|f\|_{C_B(I)},\textrm{ } \widehat{S_{n_{12},l}^{\alpha_{12}\beta_{12}}}(t_1-x_1;x_1,x_2)=0 \textrm{ and } \widehat{S_{n_{12},l}^{\alpha_{12}\beta_{12}}}(t_2-x_2;x_1,x_2)=0.
\end{eqnarray}
Let $g\in C_B^2(I)$ and $(x_1,x_2)\in I$. By the Taylor's theorem, we have
\begin{eqnarray}\label{eq4}
 \nonumber g(u_1,u_2)-g(x_1,x_2)&=&\frac{\partial g(x_1,x_2)}{\partial x_1}(u_1-x_1)+\int\limits_{x_1}^{u_1}(u_1-\alpha_1)\frac{\partial^2 g(\alpha_1,x_1)}{\partial \alpha_1^2}d\alpha_1+\frac{\partial g(x_1,x_2)}{\partial x_2}(u_2-x_2)\\
&&+\int\limits_{x_2}^{u_2}(u_2-\alpha_2)\frac{\partial^2 g(\alpha_2,x_2)}{\partial \alpha_2^2}d\alpha_2.
\end{eqnarray}
Applying the auxiliary operator defined by \ref{eq3} on both sides of \ref{eq4}, we find
\begin{eqnarray*}
\widehat{S_{n_{12},l}^{\alpha_{12}\beta_{12}}}(g;x_1,x_2)-g(x_1,x_2)&=&\widehat{S_{n_{12},l}^{\alpha_{12}\beta_{12}}}\Bigg(\int\limits_{x_1}^{u_1}(u_1-\alpha_1)\frac{\partial^2 g(\alpha_1,x_1)}{\partial \alpha_1^2}d\alpha_1;x_1,x_2\Bigg)\\
&&+\widehat{S_{n_{12},l}^{\alpha_{12}\beta_{12}}}\Bigg(\int\limits_{x_2}^{u_2}(u_2-\alpha_2)\frac{\partial^2 g(\alpha_2,x_2)}{\partial \alpha_2^2}d\alpha_2;x_1,x_2\Bigg)\\
&=&S_{n_{12},l}^{\alpha_{12}\beta_{12}}\Bigg(\int\limits_{x_2}^{u_2}(u_2-\alpha_2)\frac{\partial^2 g(\alpha_2,x_2)}{\partial \alpha_2^2}d\alpha_2;x_1,x_2\Bigg)\\
&-&\int\limits_{x_1}^{\frac{[n_1+l]x_1+\alpha_1}{[n_1]+\beta_1}}\Bigg(\frac{[n_1+l]x_1+\alpha_1}{[n_1]+\beta_1}-\alpha_1\Bigg)\frac{\partial^2 g(\alpha_1,x_1)}{\partial \alpha_1^2}d\alpha_1\\
&+&S_{n_{12},l}^{\alpha_{12}\beta_{12}}\Bigg(\int\limits_{x_2}^{u_2}(u_2-\alpha_2)\frac{\partial^2 g(\alpha_2,x_2)}{\partial \alpha_2^2}d\alpha_2;x_1,x_2\Bigg)\\
&-&\int\limits_{x_2}^{\frac{[n_2+l]x_2+\alpha_2}{[n_2]+\beta_2}}\Bigg(\frac{[n_2+l]x_2+\alpha_2}{[n_2]+\beta_2}-\alpha_2\Bigg)\frac{\partial^2 g(\alpha_2,x_2)}{\partial \alpha_2^2}d\alpha_2.
\end{eqnarray*}
Hence,
\begin{eqnarray*}
|\widehat{S_{n_{12},l}^{\alpha_{12}\beta_{12}}}(g;x_1,x_2)-g(x_1,x_2)|&\leq& S_{n_{12},l}^{\alpha_{12}\beta_{12}}\Bigg(\Big|\int\limits_{x_2}^{u_2}|u_2-\alpha_2|\Big|\frac{\partial^2 g(\alpha_2,x_2)}{\partial \alpha_2^2}\Big|d\alpha_2\Big|;x_1,x_2\Bigg)\\
&+&\int\limits_{x_1}^{\frac{[n_1+l]x_1+\alpha_1}{[n_1]+\beta_1}}\Bigg|\frac{[n_1+l]x_1+\alpha_1}{[n_1]+\beta_1}-\alpha_1\Big|\Big|\frac{\partial^2 g(\alpha_1,x_1)}{\partial \alpha_1^2}\Big|d\alpha_1\Big|\\
&+&S_{n_{12},l}^{\alpha_{12}\beta_{12}}\Big|\int\limits_{x_2}^{u_2}|u_2-\alpha_2|\Big|\frac{\partial^2 g(\alpha_2,x_2)}{\partial \alpha_2^2}\Big|d\alpha_2\Big|;x_1,x_2\Bigg)\Big|\\
&+&\Big|\int\limits_{x_2}^{\frac{[n_2+l]x_2+\alpha_2}{[n_2]+\beta_2}}\Big|\frac{[n_2+l]x_2+\alpha_2}{[n_2]+\beta_2}-\alpha_2\Big|\Big|\frac{\partial^2 g(\alpha_2,x_2)}{\partial \alpha_2^2}\Big|d\alpha_2\Big|\\
&\leq&\Big\{S_{n_{12},l}^{\alpha_{12}\beta_{12}}((u_1-x_1)^2:x_1,x_2)+\Big(\frac{[n_1+l]x_1+\alpha_1}{[n_1]+\beta_1}-x_1\Big)^2\Big\}\|g\|_{C_B^2(I)}\\
&+&\Big\{S_{n_{12},l}^{\alpha_{12}\beta_{12}}((u_2-x_2)^2:x_1,x_2)+\Big(\frac{[n_2+l]x_2+\alpha_2}{[n_2]+\beta_2}-x_2\Big)^2\Big\}\|g\|_{C_B^2(I)}\\
&=&\{\big(\delta_{n_{12},l}^{\alpha_{12}\beta_{12}}(x_1)\big)^2+\big(\delta_{n_{12},l}^{\alpha_{12}\beta_{12}}(x_2)\big)^2\}\|g\|_{C_B^2(I)}.
\end{eqnarray*}
Therefore,
\begin{eqnarray*}
 |S_{n_{12},l}^{\alpha_{12}\beta_{12}}(g;x_1,x_2)-f(x_1,x_2)|&\leq&|\widehat{S_{n_{12},l}^{\alpha_{12}\beta_{12}}}(f-g;x_1,x_2)|+|\widehat{S_{n_{12},l}^{\alpha_{12}\beta_{12}}}(g;x_1,x_2)-g(x_1,x_2)|+|g(x,y)-f(x,y)|\\
 &+&\Bigg|f\Bigg(\frac{[n_1+l]x_1+\alpha_1}{[n_1]+\beta_1},\frac{[n_2+l]x_2+\alpha_2}{[n_2]+\beta_2}\Bigg)-f(x_1,x_2)\Bigg|\\
&\leq& 3\|f-g\|_{C_B(I)}+\|f-g\|_{C_B(I)}+|\widehat{S_{n_{12},l}^{\alpha_{12}\beta_{12}}}(g;x_1,x_2)-g(x_1,x_2)|\\
&+&\Bigg|f\Bigg(\frac{[n_1+l]x_1+\alpha_1}{[n_1]+\beta_1},\frac{[n_2+l]x_2+\alpha_2}{[n_2]+\beta_2}\Bigg)-f(x_1,x_2)\Bigg|\\
&\leq& 4\|f-g\|_{C_B(I)}+\{\big(\delta_{n_{12},l}^{\alpha_{12}\beta_{12}}(x_1)\big)^2+\big(\delta_{n_{12},l}^{\alpha_{12}\beta_{12}}(x_2)\big)^2\}\|g\|_{C_B^2(I)}\\
&+&\Bigg|f\Bigg(\frac{[n_1+l]x_1+\alpha_1}{[n_1]+\beta_1},\frac{[n_2+l]x_2+\alpha_2}{[n_2]+\beta_2}\Bigg)-f(x_1,x_2)\Bigg|\\
&\leq&4\|f-g\|_{C_B(I)}+2M_{n_1,n_2}(x_1,x_2)\|_{C_B^2(I)}+\omega\Bigg(f;\sqrt{\Big(\frac{[n_1+l]x_1+\alpha_1}{[n_1]+\beta_1}\Big)^2+\Big(\frac{[n_2+l]x_2+\alpha_2}{[n_2]+\beta_2}\Big)^2}\Bigg).
\end{eqnarray*}
Next, using the equation \ref{k}, we get
\begin{eqnarray*}
|S_{n_{12},l}^{\alpha_{12}\beta_{12}}(g;x_1,x_2)-f(x_1,x_2)|&\leq& 4K(f;M_{n_1,n_2}(x_1,x_2))\\
&+&\omega\Bigg(f;\sqrt{\Big(\frac{[n_1+l]x_1+\alpha_1}{[n_1]+\beta_1}\Big)^2+\Big(\frac{[n_2+l]x_2+\alpha_2}{[n_2]+\beta_2}\Big)^2}\Bigg)\\
&\leq& M\Bigg\{\omega_2\Big(f:\sqrt(M_{n_1,n_2}(x_1,x_2))\Big)+min\{1,M_{n_1,n_2}(x_1,x_2)\}\|f\|_{C_B^2(I)}\Bigg\}\\
&+&\omega\Bigg(f;\sqrt{\Big(\frac{[n_1+l]x_1+\alpha_1}{[n_1]+\beta_1}\Big)^2+\Big(\frac{[n_2+l]x_2+\alpha_2}{[n_2]+\beta_2}\Big)^2}\Bigg).
\end{eqnarray*}
\end{proof}
Now, we discuss the degree of  approximation for the $(p,q)$-Bivariate-Schurer-Stancu operators in the Lipschitz class. We define the Lipschitz class $Lip_M(\gamma_1,\gamma_2)$ by means of two variables as follows:
\begin{eqnarray*}
|f(t_1,t_2)-f(x_1,x_2)|\leq M |t_1-x_1|^{\gamma_1}|t_2-x_2|^{\gamma_2},
\end{eqnarray*}
where $0<\gamma_1,\gamma_2\leq 1$ and for any $(t_1,t_2),(x_1,x_2)\in I\times I$.
\begin{theorem}
  Let $f\in Lip_M(\gamma_1,\gamma_2)$ and $(q_{n_i}),(p_{n_i})$, $i=1,2$ are defined in \ref{eq2}. Then for all $(x_1,x_2)\in I\times I$, we have
\begin{eqnarray*}
 |S_{n_{12},l}^{\alpha_{12}\beta_{12}}(g;x_1,x_2)-f(x_1,x_2)|\leq M \delta_{n_1}^{\gamma_1/2}(x_1)\delta_{n_2}^{\gamma_2/2}(x_2)
\end{eqnarray*}
where $\delta_{n_i}(x_i)=S_{n_{12},l}^{\alpha_{12}\beta_{12}}(((t_i-x_i)^2);p_{12},q_{12};x_1,x_2)$.
\end{theorem}
\begin{proof}
  Since $f\in Lip_M(\gamma_1,\gamma_2)$, we can write
  \begin{eqnarray*}
  |S_{n_{12},l}^{\alpha_{12}\beta_{12}}(f;q_{n_{12}},p_{n_{12}};x_1,x_2)-f(x_1,x_2)|&\leq& S_{n_{12},l}^{\alpha_{12}\beta_{12}}(|f(t_1,t_2)-f(x_1,x_2)|;q_{n_{12}},p_{n_{12}};x_1,x_2)\\
  &\leq& MS_{n_{12},l}^{\alpha_{12}\beta_{12}}(|t_1-x_1|^{\gamma_1}|t_2-x_2|^{\gamma_2};q_{n_{12}},p_{n_{12}};x_1,x_2)\\
  &=&M S_{n_{12},l}^{\alpha_{12}\beta_{12}}(|t_1-x_1|^{\gamma_1};q_{n_{12}},p_{n_{12}};x_1,x_2)S_{n_{12},l}^{\alpha_{12}\beta_{12}}(|t_2-x_2|^{\gamma_2};q_{n_{12}},p_{n_{12}};x_1,x_2).
  \end{eqnarray*}
  Next, we use the Holder inequality with $p=\frac{2}{\gamma_1}$, $q=\frac{2}{2-\gamma_1}$ and $p=\frac{2}{\gamma_1}$, $q=\frac{2}{2-\gamma_2}$, respectively, we have
  \begin{eqnarray*}
  |S_{n_{12},l}^{\alpha_{12}\beta_{12}}(f;q_{n_{12}},p_{n_{12}};x_1,x_2)-f(x_1,x_2)|&\leq&\big\{S_{n_{12},l}^{\alpha_{12}\beta_{12}}((t_1-x_1)^2;q_{n_{12}},p_{n_{12}};x_1,x_2)^{\frac{{\gamma_1}}{2}}\big\{S_{n_{12},l}^{\alpha_{12}\beta_{12}}((1;q_{n_{12}},p_{n_{12}};x_1,x_2)\big\}^{\frac{2}{2-\gamma_1}}\\
  &\times & \big\{S_{n_{12},l}^{\alpha_{12}\beta_{12}}((t_2-x_2)^2;q_{n_{12}},p_{n_{12}};x_1,x_2)\big\}^{\frac{{\gamma_2}}{2}}\big\{S_{n_{12},l}^{\alpha_{12}\beta_{12}}((1;q_{n_{12}},p_{n_{12}};x_1,x_2)\big\}^{\frac{2}{2-\gamma_2}}\\
  &=&M\delta_{n_1}^{\gamma_1/2}(x_1)\delta_{n_2}^{\gamma_2/2}(x_2).
  \end{eqnarray*}
\end{proof}

\end{document}